\documentclass[11pt,reqno]{amsart}
\usepackage{amsthm,amssymb,amsmath}

\usepackage{xcolor}

\usepackage{fix-cm}
\usepackage{enumitem,needspace}
\usepackage[normalem]{ulem}
\usepackage{esint}
\usepackage{mlmodern}
\usepackage[expansion=false]{microtype}

\usepackage[T1]{fontenc}

\usepackage[colorlinks,citecolor=blue,hypertexnames=false]{hyperref}

\AtBeginDocument{
  \hypersetup{
    urlcolor=blue,
    citecolor=blue,
    linkcolor=blue,
  }
}

\newtheorem{theorem}{Theorem}
\newtheorem{proposition}{Proposition}
\newtheorem*{theorem*}{Theorem}
\newtheorem{lemma}{Lemma}[section]

\theoremstyle{definition}

\newtheorem*{definition*}{\bf Definition}

\newtheorem{remark}{\sc Remark}[section]
\newtheorem*{remark*}{\sc Remark}

\newtheorem*{example*}{\bf Example}

\newcommand{\R}{\mathbb R}
\newcommand{\dd}{\,\mathrm d}
\newcommand{\e}{\mathrm e}

\numberwithin{equation}{section}

\hypersetup{
  pdftitle={Lower bounds for  heat kernels of elliptic operators with unbounded diffusion, drift, and potential terms},
  pdfauthor={Sallah Eddine Boutiah},
  pdfsubject={Second revised manuscript with changes in red}
}

\begin{document}

\title[A lower bound for heat kernels]{Lower bounds for heat kernels of elliptic operators with unbounded diffusion, drift, and potential terms}
\author{Sallah Eddine Boutiah}



\begin{abstract}
We establish a pointwise lower bound for the heat kernel of the elliptic operator $
\Lambda=(1+|x|^\alpha)\Delta +b|x|^{\alpha-2}x\cdot\nabla-|x|^\beta$, where $d\geq3$, $\alpha>2$, $\beta>\alpha-2$, and $b\in\mathbb R$.
For every $\tau>0$, we prove that 
$$
\begin{aligned}
p(t,x,y) &\geq C_\tau e^{\lambda_0t} \left(\frac{1+|y|^\alpha}{1+|x|^\alpha}\right)^{\frac{b}{2\alpha}}
\frac{(|x||y|)^{-\frac{d-1}{2}-\frac{\beta-\alpha}{4}}}{1+|y|^\alpha}\\
&\quad\times\exp\!\left[ -\int_1^{|x|}\sqrt{\frac{s^\beta}{1+s^\alpha}}\,\mathrm ds -\int_1^{|y|}\sqrt{\frac{s^\beta}{1+s^\alpha}}\,\mathrm ds \right]
\end{aligned}
$$
for all $t\geq\tau$ and $|x|,|y|\geq1$, where $\lambda_0<0$ is the largest eigenvalue of $\Lambda$ and $C_\tau>0$ is  independent of $t,x,y$. The proof applies the classical Davies–Simon argument in a weighted symmetric setting.
\end{abstract}

\address{Universit\'{e} Laval, D\'{e}partement de math\'{e}matiques et de statistique, Qu\'{e}bec, QC, Canada and Laboratoire de Math\'{e}matiques Appliqu\'{e}es, Universit\'{e} Ferhat Abbas, S\'{e}tif 1, Campus El Bez, S\'{e}tif, Algeria}
\email{sallah-eddine.boutiah.1@ulaval.ca}
\fontsize{10.5pt}{4.4mm}\selectfont
\maketitle
\section{Introduction and main result}
\label{sec:introduction}
Consider the following elliptic operator
\begin{equation}\label{def-A}
\Lambda u(x) = (1+|x|^{\alpha})\Delta u(x) + b|x|^{\alpha-2}x\cdot\nabla u(x) - |x|^{\beta}u(x),
\qquad x\in\mathbb{R}^{d},
\end{equation}
where $d\geq 3$, $\alpha>2$, $\beta>\alpha-2$, and $b\in\mathbb{R}$.

This paper continues \cite{BTR} by establishing a lower bound for the heat kernel of $\Lambda$ with the same spatial decay as the upper estimate proved there. The proof applies the Davies–Simon principle to the weighted symmetric realization of $\Lambda$, using the ground-state estimates and intrinsic upper bound established in \cite{BTR}. The resulting estimate holds for arbitrary pairs of spatial points and is uniform in time on every interval $[\tau,\infty)$, with $\tau>0$.

We first specify the realization of $\Lambda$ and recall the
results needed to state the main theorem. For $1<p<\infty$, let
$\Lambda_p$ denote the realization in $L^p(\mathbb R^d)$ with domain

\[
\begin{aligned}
D_p(\Lambda)
=
\bigl\{
u\in W^{2,p}(\mathbb{R}^{d}) :\;&
(1+|x|^\alpha)|D^2u|\in L^p(\mathbb{R}^{d}),\\
&
(1+|x|^{\alpha-1})|\nabla u|\in L^p(\mathbb{R}^{d}),\\
&
|x|^\beta u\in L^p(\mathbb{R}^{d})
\bigr\}.
\end{aligned}
\]
It was proved in \cite{boutiah et al} that $\Lambda_p$ generates
a positive strongly continuous analytic semigroup $(T_p(t))_{t\geq0}$.
These semigroups are consistent for $1<p<\infty$.
Moreover, for every $t>0$ and $\nu\in(0,1)$,
$$
T_p(t)L^p(\mathbb{R}^{d})
\subset C_b^{1+\nu}(\mathbb{R}^{d}),
$$
and the semigroup is immediately compact; see
\cite[Propositions 5 and 6]{boutiah et al}.

The spectrum $\sigma(\Lambda_p)$ is independent of $p$ and consists
of a sequence of negative real eigenvalues accumulating at $-\infty$ and  $(T_p(t))_{t\geq 0}$ is irreducible;
see \cite[Section~2]{BTR}. The eigenspace associated
with the largest eigenvalue $\lambda_0<0$ is one-dimensional
and is spanned by a strictly positive radial eigenfunction $\Phi$.
Furthermore, $\Phi\in C_b^{1+\nu}(\mathbb{R}^{d})\cap C^2(\mathbb{R}^{d})$ for every $\nu\in(0,1)$ 
and $\Phi(x)\downarrow0$ as $|x|\to\infty$.

The semigroups admit a common nonnegative integral kernel $p$ such that
\[
    T_p(t)f(x)=\int_{\mathbb R^d}p(t,x,y)f(y)\, dy,
    \qquad t>0,
\]
for every $f\in L^p(\mathbb R^d)$ and almost every $x\in\mathbb R^d$,
cf. \cite{Lo-Be,Me-Pa-Wa}.
All pointwise statements below refer to the jointly continuous version
constructed in Proposition~\ref{prop:kernel}. The  weight 
$$
 \rho(x):=(1+|x|^\alpha)^{\frac b\alpha-1}
$$
allows us to write $\Lambda$ in divergence form
$$
 \Lambda u =\frac1\rho\operatorname{div}\bigl((1+|x|^\alpha)\rho\nabla u\bigr) -|x|^\beta u.
$$
As shown in Proposition~\ref{prop:kernel}, the kernel satisfies the weighted symmetry relation $\rho(x)p(t,x,y)=\rho(y)p(t,y,x)$.
Equivalently, the kernel with respect to $\rho(x)\,dx$,
$
    k(t,x,y):=p(t,x,y) / \rho(y),
$
is symmetric in $x$ and $y$. Since $p(t,x,y)=k(t,x,y)\rho(y)$, estimates for $k$ must be
multiplied by $\rho(y)$ to obtain the corresponding estimates for the kernel $p$ with respect to Lebesgue measure.

To describe the decay of $\Phi$, set

\begin{equation}\label{eq:pI}
\gamma:=\frac{d-1}{2}+\frac{\beta-\alpha}{4},
\qquad
I(r):=\int_1^r
\sqrt{\frac{s^\beta}{1+s^\alpha}}\, ds,
\qquad r\geq 1.
\end{equation}

The ground-state estimate proved in \cite{BTR} is the following.

\begin{proposition}[{\cite[Proposition 1]{BTR}}]
\label{baondedofphi}
Let $d\geq 3$, $\alpha>2$, $\beta>\alpha-2$, and
$b\in\mathbb{R}$. Let $\Phi$ be a strictly positive eigenfunction
of $\Lambda$ corresponding to its largest eigenvalue
$\lambda_0<0$. Then there exist constants $c_-,c_+>0$ such that
\begin{equation}\label{eq:phicomparison}
\begin{aligned}
c_-|x|^{-\gamma}
(1+|x|^\alpha)^{-\frac{b}{2\alpha}}
\exp\bigl(-I(|x|)\bigr)
&\leq \Phi(x)
\\
&\leq
c_+|x|^{-\gamma}
(1+|x|^\alpha)^{-\frac{b}{2\alpha}}
\exp\bigl(-I(|x|)\bigr)
\end{aligned}
\end{equation}
for every $|x|\geq 1$.
\end{proposition}

\medskip

At the level of the heat kernel, the intrinsic ultracontractivity
established in \cite{BTR} yields, for each $t>0$,
\begin{equation}\label{eq:intro-upper}
 p(t,x,y)\leq m_t\e^{\lambda_0t}\Phi(x)\Phi(y)\rho(y),
 \qquad x,y\in\mathbb R^d,
\end{equation}
where $m_t>0$ is independent of $x,y$.
For $0<t\leq1$, one may take
\[
 m_t=C_1\exp\bigl(C_2t^{-\eta}-\lambda_0t\bigr),
 \qquad
 \eta=\frac{\beta-\alpha+2}{\beta+\alpha-2}\in(0,1).
\]
Moreover, a constant admissible at time $\tau$ remains admissible
for all $t\geq\tau$; this follows from the semigroup property and
the ground-state identity \eqref{equa_Phi}, as recalled in Lemma~\ref{lem:upper}.

A diagonal lower estimate is also obtained in \cite[Proposition~3]{BTR}. Here we derive the corresponding lower
bound for arbitrary pairs of spatial points by combining the intrinsic upper estimate with the Davies-Simon argument.

\medskip

We now state the main result of this paper.

\begin{theorem}\label{thm:main}
Let $d\geq 3$, $\alpha>2$, $\beta>\alpha-2$, and
$b\in\mathbb R$. For every $\tau>0$, there exists a constant
$C_\tau>0$ such that, for all $t\geq\tau$ and $|x|,|y|\geq1$,
\begin{equation}\label{eq:main}
 p(t,x,y)\geq C_\tau\e^{\lambda_0t}
 \left(\frac{1+|y|^\alpha}{1+|x|^\alpha}\right)^{\frac{b}{2\alpha}}
 \frac{(|x||y|)^{-\gamma}}{1+|y|^\alpha}
 \exp\bigl(-I(|x|)-I(|y|)\bigr).
\end{equation}
The constant $C_\tau$ is independent of $t,x,y$.
\end{theorem}

In fact, we prove the global intrinsic estimate
\begin{equation}\label{eq:main-intrinsic}
 p(t,x,y)\geq\widetilde C_\tau\e^{\lambda_0t}
 \Phi(x)\Phi(y)\rho(y),
 \qquad t\geq\tau,\quad x,y\in\mathbb R^d,
\end{equation}
for some $\widetilde C_\tau>0$ independent of $t,x,y$.
Together with \eqref{eq:intro-upper}, this gives
\begin{equation}\label{eq:two-sided}
 \widetilde C_\tau
 \leq\frac{\e^{-\lambda_0t}p(t,x,y)}{\Phi(x)\Phi(y)\rho(y)}
 \leq m_\tau,
 \qquad t\geq\tau,\quad x,y\in\mathbb R^d.
\end{equation}
Thus the spatial term in \eqref{eq:main} is sharp up to
multiplicative constants, by \eqref{eq:phicomparison}.
The restriction $|x|,|y|\geq1$ enters only through the explicit
formula for the ground-state profile. The constants in
\eqref{eq:two-sided} depend on the chosen normalization of $\Phi$. 

\noindent
\textbf{Proof strategy.}
We prove \eqref{eq:main-intrinsic} by applying the argument of Davies and Simon \cite[Theorem~3.2]{DS} to the symmetric
kernel $p(t,x,y)/\rho(y)$. For fixed $s>0$, the intrinsic upper bound and the integrability
of $\Phi^2\rho$ yield a localization estimate on a ball whose radius is independent of the spatial variables.
Continuity and strict positivity on this ball, combined with the semigroup property, then give a global intrinsic lower
bound at time $3s$.

Taking $s=\tau/3$ and using the ground-state identity, we extend this bound to all $t\geq\tau$ with the same constant.
The explicit spatial estimate follows from the ground-state bounds in Proposition~\ref{baondedofphi}.

\begin{remark}
When $b=0$, our theorem provides a lower bound for the heat kernel of the Schr\"odinger-type operator
$$(1+|x|^\alpha)\Delta-|x|^\beta.$$
The corresponding upper bound was established in \cite{CRT17}.
\end{remark}

\subsection{Related results}\label{subsec:literature-kernel-estimates}
There is an extensive literature on heat kernel bounds for local and non-local operators with unbounded coefficients or 
having critical polar singularities that make the standard heat kernel bounds invalid. 
For unbounded diffusion coefficients, generation
and regularity results are developed in \cite{Me-Pa-Wa,FL}. In the absence of a drift and a potential, Metafune and Spina
\cite{MS12b} derive kernel estimates from weighted Nash and Hardy
inequalities. More precisely, for $d\geq3$ and $\alpha>2$, if
$q$ denotes the kernel of $(1+|x|^\alpha)\Delta$ relative to $d\mu(x)=(1+|x|^\alpha)^{-1}\, d x$, then

\begin{equation}
 q(t,x,y)\leq C t^{-\sigma}\phi(x)\phi(y),
 \qquad 0<t\leq1,\qquad
 \sigma=
 \begin{cases}
  \dfrac{d+\alpha-4}{\alpha-2},&2<\alpha\leq4,\\[4pt]
  \dfrac d2,&\alpha>4,
 \end{cases}
 \label{eq:literature-pure-diffusion}
\end{equation}
where the positive ground state satisfies\footnote{Here $\asymp$ means comparison up
to positive constants independent of $x$.}
$\phi(x)\asymp(1+|x|)^{2-d}$. The corresponding Lebesgue
kernel is $q(t,x,y)/(1+|y|^\alpha)$. The low-dimensional model
$(1+|x|^2)^{\alpha/2}\Delta$ is treated by Spina
\cite{Sp13}, while Metafune and Spina \cite{MS14}
investigate the homogeneous, degenerate operator $|x|^\alpha\Delta$.

For Schr\"odinger operators with unbounded diffusion and potential
terms Lorenzi and Rhandi \cite{LR} consider
$(1+|x|^\alpha)\Delta-|x|^\beta$, establish generation for
$0\leq\alpha\leq2$, $\beta\geq0$, and derive heat kernel
estimates for $0\leq\alpha<2$, $\beta\geq2$.
For $\alpha>2$ and $\beta>\alpha-2$, the $L^p$ realization
and ground-state estimates are developed by Canale, Rhandi
and Tacelli in \cite{CRT16,CRT17}. Canale and Tacelli
\cite{CT} obtain weighted bounds with polynomial time
dependence. The critical case $\beta=\alpha-2$ is treated
by Durante, Manzo and Tacelli \cite{DMT}.

For nonzero $b$ and no potential, Metafune, Spina and Tacelli
\cite{MST14a} study
$(1+|x|^\alpha)\Delta+b|x|^{\alpha-2}x\cdot\nabla$.
The addition of $-|x|^\beta$, with $\beta>\alpha-2$, leads
to the realizations studied in \cite{boutiah et al} and the
kernel estimates of \cite{BTR} used above. The present proof
uses their intrinsic upper bound as an input to obtain the
off-diagonal lower estimate \eqref{eq:main-intrinsic}.
For more general Kolmogorov operators, estimates based on
time-dependent Lyapunov functions are obtained in
\cite{Sp08,KLR}.

Weighted estimates also arise for operators with singular drifts.
Milman and Sem\"enov \cite{MS04} develop a desingularization
method based on Sobolev estimates and weighted $L^1$ semigroup
bounds. Metafune, Sobajima and Spina \cite{MSS} treat
discontinuous diffusion coefficients, radial singular drifts and
inverse-square potentials, while Kinzebulatov, Sem\"enov and
Szczypkowski \cite{KSS} establish two-sided estimates for the
fractional Laplacian with a critical Hardy drift.

Recently, Boutiah and Kinzebulatov \cite{BK} obtain weighted upper heat kernel bounds for attractive particle systems with finite number of particles $N\ge2$ and $d\ge3$, when the attraction strength satisfies $0<\nu<2(d-2)$. Their proof uses weighted Sobolev estimates, Moser iteration and the many-particle Hardy inequality \cite{HHLT}. In dimension two, they also establish upper heat kernel bounds for a finite-particle system of Keller–Segel type under a suitable smallness assumption on the attraction strength.

Boutiah, Kinzebulatov and Madou \cite{BKM}
derive related lower bounds from spectral-gap estimates
for Gaussian measures with logarithmic interaction weights.

\section{Weighted estimates and pointwise kernel properties}
\label{sec:preliminaries}

For the weight $\rho$, we use the notation

\[
 L_\rho^2:=L^2(\mathbb R^d,\rho(x)\,\dd x),
 \qquad \langle g\rangle_\rho:=\int_{\mathbb R^d}g(x)\rho(x)\,\dd x,
 \qquad \langle g\rangle:=\int_{\mathbb R^d}g(x)\,\dd x.
\]
\begin{lemma}\label{lem:tail}
The ground state has finite, strictly positive weighted mass:

\begin{equation}\label{bound_varphi}
    0 < \langle |\Phi|^2 \rangle_\rho < \infty.
\end{equation}
\end{lemma}

\begin{proof}
By \eqref{eq:phicomparison},
\[
    \Phi(x)^2\rho(x)
    \leq c_+^2
    \frac{|x|^{-2\gamma}}{1+|x|^\alpha}
    \e^{-2I(|x|)},
    \qquad |x|\geq 1.
\]
Since $I(r)\geq0$ and $1+r^\alpha\geq r^\alpha$ for $r\geq1$,
and $d-1-2\gamma-\alpha=-(\alpha+\beta)/2$, polar coordinates give,
for every $R\geq1$,
\begin{align*}
\langle \Phi^2 \mathbf{1}_{|x|>R} \rangle_{\rho}
    &\leq
    \omega_{d-1}c_+^2
    \int_R^\infty r^{-(\alpha+\beta)/2}\,\dd r \\
    &=
    \frac{2\omega_{d-1}c_+^2}{\alpha+\beta-2}
    R^{-(\alpha+\beta-2)/2},
\end{align*}
Here $\omega_{d-1}=|\mathbb S^{d-1}|$, and the integral converges
because $\alpha+\beta-2>2(\alpha-2)>0$. 

On the other hand, the continuity and positivity of
$\Phi$ and $\rho$ imply
$$
    0< \langle \Phi^2 \mathbf{1}_{B_1} \rangle_{\rho}   \leq |B_1|\max_{\overline{B_1}}(\Phi^2\rho)
    <\infty.
$$
Combining these two estimates with $R=1$ proves \eqref{bound_varphi}.

\end{proof}

The following consequence of \cite{BTR} is the only global kernel
estimate required in the lower-bound argument. We also record its
stability under evolution to later times.

\begin{lemma}\label{lem:upper}
For each $t>0$, there is $m_t>0$ such that, for almost every $(x,y)\in\mathbb R^d\times\mathbb R^d$,
\begin{equation}\label{eq:upper}
    0 \leq p(t,x,y)
    \leq m_t\e^{\lambda_0t}\Phi(x)\Phi(y)\rho(y).
\end{equation}
Moreover, any constant $m_r$ admissible at time $r>0$
remains admissible for all $t\geq r$.

\end{lemma}

\begin{proof}

Set $a(x)=1+|x|^\alpha$, $\varphi(x)=a(x)^{b/(2\alpha)}$ and
$\eta=(\beta-\alpha+2)/(\beta+\alpha-2)$.
Let $ d\mu_0(x)=a(x)^{-1}\, dx$ and let $k_{\mu_0}$ denote the
weighted kernel used in \cite{BTR}.
By \cite[Section~3, Lemma~2]{BTR},
for every $t>0$ and almost every $(x,y)$,
\[
    p(t,x,y)=\frac{\varphi(y)}{\varphi(x)a(y)}k_{\mu_0}(t,x,y).
\]
Moreover, \cite[Theorem~1]{BTR} gives, for $0<t\leq1$,
\[
    k_{\mu_0}(t,x,y)
    \leq C_1\exp(C_2t^{-\eta})(\varphi\Phi)(x)(\varphi\Phi)(y).
\]

Since $\varphi^2/a=\rho$, we obtain
\[
    p(t,x,y)
    \leq C_1\exp(C_2t^{-\eta})\Phi(x)\Phi(y)\rho(y).
\]

Thus, for $0<t\leq1$, \eqref{eq:upper} holds with

\[
    m_t
    = C_1\exp\!\left(
        C_2t^{-\frac{\beta-\alpha+2}{\beta+\alpha-2}}
        -\lambda_0t
    \right).
\]
To extend this estimate, suppose that \eqref{eq:upper} holds at some time
$r>0$. Fix $q>d$. The ground-state estimate
\eqref{eq:phicomparison} and $\beta>\alpha-2$ give
$\Phi\in L^q(\mathbb R^d)$. Since $\Phi\in C^2(\mathbb R^d)$
and $\Lambda\Phi=\lambda_0\Phi$, the maximal domain characterization
in \cite{boutiah et al} yields $\Phi\in D_q(\Lambda)$.
Hence $T_q(s)\Phi=\e^{\lambda_0s}\Phi$, and the kernel representation yields
\begin{equation}\label{equa_Phi}
    \int_{\R^d}p(s,x,z)\Phi(z)\,\dd z
    =\e^{\lambda_0s}\Phi(x),\qquad s>0,
\end{equation}
for almost every $x\in\R^d$.

For every $t>r$, the semigroup property, Tonelli's theorem
and uniqueness of kernel representations give the following estimate for almost every $(x,y)\in\mathbb R^d\times\mathbb R^d$,
where we use \eqref{equa_Phi}:
\begin{align*}
    p(t,x,y)
    &= \int_{\R^d}p(t-r,x,z)p(r,z,y)\,\dd z \\
    &\leq m_r\e^{\lambda_0r}\Phi(y)\rho(y)
        \int_{\R^d}p(t-r,x,z)\Phi(z)\,\dd z \\
    &= m_r\e^{\lambda_0t}\Phi(x)\Phi(y)\rho(y).
\end{align*}
Thus, $m_r$ remains admissible for every $t\geq r$.
In particular, choosing $r=1$ establishes \eqref{eq:upper}
for all $t>1$ with $m_t=m_1$.
\end{proof}
The estimates above initially hold almost everywhere. The next
proposition selects a single continuous version of the kernel and
establishes the identities at every spatial point. This also
provides the compact-set positivity used in
Proposition~\ref{prop:bridge}.
\begin{proposition}\label{prop:kernel}
The kernel $p$ admits a version, still denoted by $p$, such that
\begin{equation}\label{eq:regularpositive}
    p\in C\bigl((0,\infty)\times\R^d\times\R^d\bigr),
    \qquad p(t,x,y)>0.
\end{equation}
For this version, the following properties hold for every $t>0$ and $x,y\in\R^d$:
\begin{enumerate}[label=\textup{(\roman*)},leftmargin=2em]
\item $\rho(x)p(t,x,y)=\rho(y)p(t,y,x)$ \\
\item $\int_{\R^d}p(t,x,y)\Phi(y)\,dy   =\e^{\lambda_0t}\Phi(x)$ \\
\item $p(t,x,y)\leq m_t\e^{\lambda_0t}\Phi(x)\Phi(y)\rho(y)$, with the same constant $m_t$ as in Lemma~\ref{lem:upper}. \\
\item For every $t,s>0$,
\begin{equation}\label{eq:semigroup}
    p(t+s,x,y)=\int_{\R^d}p(t,x,z)p(s,z,y)\, dz.
\end{equation}
\end{enumerate}
\end{proposition}

\begin{proof}
Set 
$$
 a(x)=1+|x|^\alpha,\qquad V(x)=|x|^\beta,\qquad
 \rho(x)=a(x)^{b/\alpha-1},\qquad d\mu(x)=\rho(x)\,dx.
$$
Since $a\rho=a^{b/\alpha}$, then, 
$\nabla(a\rho)=b|x|^{\alpha-2}x\rho$. So,
\begin{equation}\label{eq:weighted-divergence}
 \Lambda f=\rho^{-1}\operatorname{div}(a\rho\nabla f)-Vf.
\end{equation}

\smallskip
\noindent\emph{Step 1.}
Set $B_n=B(0,n)$. On each $\overline{B_n}$, the coefficients of
$\Lambda$ are H\"older continuous and bounded, $a\ge1$, and $\rho$
is bounded above and below by positive constants, since $\alpha>2$
and $\beta>\alpha-2>0$.

Let $p_n$ be the kernel of the Dirichlet semigroup generated by $\Lambda$ on $B_n$.
Its existence, spatial continuity and strict positivity follow from
\cite[Corollary~3.5 and Section~4.1]{AEG}.
Indeed,  the operator can be written as
$$
 -\Lambda=-\operatorname{div}(a\nabla)
       +(\alpha-b)|x|^{\alpha-2}x\cdot\nabla+V,
$$
For each fixed $n$, all coefficients are real and bounded on
$B_n$, and the principal part is uniformly elliptic since $a\geq1$, so the Dirichlet realization of $-\Lambda$ in $L^2(B_n)$ therefore falls within the framework of \cite[Section~4.1]{AEG}. The same holds for its adjoint, since the adjoint form has the same principal part and bounded real lower-order coefficients. Then, \cite[Corollary~3.5]{AEG} gives, for every $t>0$,
$$
p_n(t,\cdot,\cdot)\in C_0(B_n\times B_n),
\qquad
p_n(t,x,y)>0
\quad\text{for all }x,y\in B_n.
$$
Put $k_n(t,x,y)=p_n(t,x,y)/\rho(y)$. By \eqref{eq:weighted-divergence}, the Dirichlet realization of
$-\Lambda$ in $L_\mu^2(B_n)$ is associated with
$$
\mathfrak a_n(f,g) = \int_{B_n} \bigl(a\nabla f\cdot\nabla\overline g+Vf\overline g\bigr)\,d\mu, \qquad D(\mathfrak a_n)=H_0^1(B_n).
$$
Since $a\geq1$, $V\geq0$, and $a,V,\rho,\rho^{-1}$ are bounded on $B_n$, then $\mathfrak a_n$ is densely defined, nonnegative, symmetric and closed, with form norm equivalent to the $H_0^1(B_n)$ norm. Its semigroup is therefore self-adjoint \cite[Propositions~1.24 and~1.51]{O}. The semigroup associated with $\mathfrak a_n$ coincides with the Dirichlet semigroup. Its kernel with respect to $\mu$ is therefore $k_n$.
Thus, by spatial continuity,
$$
k_n(t,x,y)=k_n(t,y,x),
\qquad t>0,\quad x,y\in B_n.
$$
We next show that the kernels increase with the domain. Let $u_n$ and $u_{n+1}$ be the Dirichlet solutions on $B_n$
and $B_{n+1}$ with the same nonnegative initial datum
$f\in C_c^\infty(B_n)$, extended by zero.
Their difference $w=u_n-u_{n+1}$ satisfies
\[
\begin{cases}
(\partial_t-\Lambda)w=0
    & \text{in }(0,\infty)\times B_n,\\
w=-u_{n+1}\leq0
    & \text{on }(0,\infty)\times\partial B_n,\\
w(0,\cdot)=0
    & \text{in }B_n.
\end{cases}
\]
Since $V\geq0$, the parabolic maximum principle gives
$w\leq0$, hence $u_n\leq u_{n+1}$. Using the kernel representations,  $f\geq0$ and continuity, we obtain
$p_n\leq p_{n+1}$ pointwise. Together with the strict positivity,
this yields
\begin{equation}\label{eq:domain-monotonicity}
0<p_n(t,x,y)\leq p_{n+1}(t,x,y),
\qquad t>0,\quad x,y\in B_n.
\end{equation}

\smallskip

\noindent\emph{Step 2. } Fix $q>d$, $T_0>0$ and $0\leq f\in C_c^\infty(\mathbb R^d)$.
Set $u(t)=T_q(t)f$. By
\cite[Section~2 and Theorem~2]{boutiah et al}, resolvent consistency and uniqueness of the Laplace transform
identify $u$ with the solution given by the $C_0$ semigroup. Thus $ u\in C\bigl([0,T_0];C_0(\mathbb R^d)\bigr),
$ and $u$ is a nonnegative classical solution of $\partial_tu=\Lambda u$ for $t>0$. We claim that
$$
\varepsilon_n:=
\sup_{\substack{0\leq t\leq T_0\\ |x|\geq n}}u(t,x) \rightarrow0.
$$
Otherwise, after passing to a subsequence, there would exist
$t_j\to t\in[0,T_0]$ and $|x_j|\to\infty$ such that
$u(t_j,x_j)\geq\varepsilon>0$. This contradicts
$$
u(t_j,x_j)
\leq \|u(t_j,\cdot)-u(t,\cdot)\|_\infty+|u(t,x_j)|
\rightarrow0.
$$
For $\operatorname{supp}f\subset B_n$, let
$$
u_n(t,x)=\int_{B_n}p_n(t,x,y)f(y)\,dy,
\qquad t>0,
$$
with $u_n(0,\cdot)=f$.
The difference $u-u_n$ solves the same homogeneous parabolic
equation, with zero initial data and boundary values in
$[0,\varepsilon_n]$. Since
$(\partial_t-\Lambda)\varepsilon_n=V\varepsilon_n\geq0$,
the comparison principle gives
\begin{equation}\label{eq:exhaustion-comparison}
0\leq u(t,x)-u_n(t,x)\leq\varepsilon_n,
\qquad 0\leq t\leq T_0,\quad x\in B_n.
\end{equation}

\smallskip

\noindent\emph{Step 3.} Fix $r>0$ and $t\geq r$. For every nonnegative
$f\in C_c^\infty(B_n)$, comparison, the kernel
representations and Lemma~\ref{lem:upper} give
\[
\int_{B_n}p_n(t,x,y)f(y)\,dy
\leq
m_r e^{\lambda_0t}\Phi(x)
\int_{B_n}\Phi(y)\rho(y)f(y)\,dy
\]
for almost every $x\in B_n$. Since $f$ has compact support and $p_n(t,\cdot,\cdot)$
and $\Phi$ are continuous, both sides are continuous in $x$. Hence the inequality holds for every $x\in B_n$.

Fixing $x$ and using the arbitrariness of $f\geq0$,
together with continuity in $y$, yields
\[
p_n(t,x,y)
\leq m_r e^{\lambda_0t}\Phi(x)\Phi(y)\rho(y),
\qquad x,y\in B_n.
\]
Dividing by $\rho(y)>0$, we obtain
\begin{equation}\label{eq:local-kernel-bound}
0\leq k_n(t,x,y)
\leq m_r e^{\lambda_0t}\Phi(x)\Phi(y),
\qquad t\geq r>0,\quad x,y\in B_n.
\end{equation}
The argument applies to each fixed time, so the estimate holds for every $t\geq r$.

Since $\rho>0$, \eqref{eq:domain-monotonicity} also gives
$k_n\leq k_{n+1}$ on $B_n\times B_n$.
For fixed $t>0$ and $x,y\in\mathbb R^d$, the sequence
$k_n(t,x,y)$ is therefore increasing once $n$ is large
enough that $x,y\in B_n$, and
\eqref{eq:local-kernel-bound} bounds it independently of $n$.
Consequently,
$$
k(t,x,y):=\lim_{n\to\infty}k_n(t,x,y)
$$
exists and is finite for every $t>0$ and $x,y\in\mathbb R^d$.

\smallskip

\noindent\emph{Step 4.}  Fix $0<\tau<T$ and $R>0$. By
\eqref{eq:local-kernel-bound}, the kernels $k_n$ are uniformly
bounded for $n>R+1$ on $ [\tau/2,T+1]\times\overline{B_{R+1}} \times\overline{B_{R+1}}$.
The semigroup property and spatial continuity give
\[
k_n(t,x,y)
=\int_{B_n}p_n(t-s,x,z)k_n(s,z,y)\,dz,
\qquad t>s>0.
\]

For fixed $s>0$ and $y\in B_n$, the RHS is the Dirichlet solution with initial datum $k_n(s,\cdot,y)\in C_0(B_n)\subset L^2(B_n)$ at time $s$. Hence, $k_n(\cdot,\cdot,y)$ is a weak solution of
$\partial_t k_n=\Lambda k_n$ for positive times. Strong continuity on $C_0(B_n)$ also gives continuity
in $(t,x)$; see \cite[Section~4.1]{AEG}. So, 
$$
\partial_t w=\operatorname{div}(a\nabla w) +(b-\alpha)|x|^{\alpha-2}x\cdot\nabla w-Vw
$$
has bounded coefficients and is uniformly parabolic on the larger cylinder. Interior H\"older estimates
\cite[``H\"older continuity'']{Aro16},
combined with symmetry, therefore yield
\[
|k_n(t,x,y)-k_n(s,x',y')|
\leq C\bigl(|t-s|^{\theta/2}
          +|x-x'|^\theta+|y-y'|^\theta\bigr)
\]
for $t,s\in[\tau,T]$ and
$x,x',y,y'\in\overline{B_R}$, with $C>0$ and
$0<\theta<1$ independent of $n$.
Passing to the limit proves that $k$ is jointly continuous
and symmetric.

Set $p(t,x,y)=\rho(y)k(t,x,y)$ and extend $p_n$ by zero
outside $B_n\times B_n$. Monotone convergence and
\eqref{eq:exhaustion-comparison} give
\[
\int_{\mathbb R^d}p(t,x,y)f(y)\,dy=T_q(t)f(x),
\qquad t>0,\quad x\in\mathbb R^d,
\]
for every $0\leq f\in C_c^\infty(\mathbb R^d)$.
Uniqueness of locally integrable kernels implies
$p(t,\cdot,\cdot)=p_0(t,\cdot,\cdot)$ almost everywhere
for each $t>0$.
Thus $p$ is a continuous version of the original kernel.
Moreover, $p(t,x,y)\geq p_n(t,x,y)>0$ whenever $x,y\in B_n$,
which proves \eqref{eq:regularpositive}.
Symmetry gives \textup{(i)}, while passing to the limit in
\eqref{eq:local-kernel-bound} with $r=t$ gives
\textup{(iii)} with the same constant $m_t$.

\smallskip

\noindent\emph{Step 5.} Fix $0<r<T<\infty$ and $R>0$, and consider $ t,t_1,t_2\in[r,T]$, and  $x,y\in\overline{B_R}$. 
By Lemma~\ref{lem:upper}, the constant $m_r$ is admissible
in each kernel estimate. Since $\Phi$ is bounded on
$\overline{B_R}$, the upper bound established above gives
\[
\begin{aligned}
0\leq k(t,x,z)\Phi(z) &\leq C\Phi(z)^2,\\
0\leq k(t_1,x,z)k(t_2,z,y) &\leq C\Phi(z)^2,
\end{aligned}
\qquad z\in\mathbb R^d,
\]
where $C$ depends only on $r,T,R$ and the fixed data. By Lemma~\ref{lem:tail}, $\Phi^2\in L^1_\mu$.
Joint continuity of $k$ and dominated convergence therefore show that the two integrals below (\eqref{two_int}) are finite and jointly
continuous in $(t,x)$ and $(t_1,t_2,x,y)$, respectively.
Since $r,T,R$ are arbitrary, this holds for all positive times
and all spatial variables.

The ground-state relation \eqref{equa_Phi} and the semigroup
property give, respectively,
\begin{equation}\label{two_int}
\int_{\mathbb R^d}k(t,x,z)\Phi(z)\,d\mu(z) =e^{\lambda_0t}\Phi(x), \quad
\int_{\mathbb R^d}k(t_1,x,z)k(t_2,z,y)\,d\mu(z) =k(t_1+t_2,x,y),
\end{equation}
initially almost everywhere in the spatial variables,
for each fixed choice of positive times. The first equality
uses the kernel representation, while the second follows
from Tonelli's theorem and uniqueness of kernel representations.
Continuity in the spatial variables extends both equalities
to every point.

Finally, substituting $p(t,x,y)=\rho(y)k(t,x,y)$ and
$d\mu(y)=\rho(y)\,dy$ proves \textup{(ii)} and \textup{(iv)}.

\end{proof}

We next derive a lower bound from the intrinsic upper estimate. The following lemma shows that, for each fixed time, restricting the ground-state integral to a sufficiently large ball preserves at least half of its value. The radius of this ball is independent of the spatial variables.

\begin{lemma}[Uniform localization]\label{lem:mass}
For every $s>0$, there is $R\geq1$ such that
\begin{align}
\int_{\overline B_R} p(s,x,z)\Phi(z) dz
&\geq\frac12\e^{\lambda_0s}\Phi(x),
&&x\in\R^d, \label{eq:initialmass}\\
\int_{\overline B_R}\Phi(w)\rho(w)p(s,w,y) d w
&\geq\frac12\e^{\lambda_0s}\Phi(y)\rho(y),
&&y\in\R^d. \label{eq:terminalmass}
\end{align}
Both integrals are finite; \(R\) is independent of \(x,y\).
\end{lemma}

\begin{proof}
We use the version of $p$ provided by
Proposition~\ref{prop:kernel}.
Fix $s>0$, and let $m_s$ be the constant from
Lemma~\ref{lem:upper}. By Lemma~\ref{lem:tail}, we may choose $R\geq1$ so large that

\begin{equation}\label{eq:tailchoice}
    m_s  \langle \Phi(\cdot)^2  \mathbf{1}_{\mathbb{R}^d\backslash \overline B_R} \rangle_\rho \leq\frac{1}{2}.
\end{equation}

By Proposition~\ref{prop:kernel}\textup{(ii)},
\[
    \int_{\R^d}p(s,x,z)\Phi(z)\,d z
    =\e^{\lambda_0s}\Phi(x)<\infty,
    \qquad x\in\R^d.
\]

Moreover, Proposition~\ref{prop:kernel}\textup{(iii)}
and \eqref{eq:tailchoice} give
\begin{align*}
    \int_{\mathbb{R}^d\backslash \overline B_R} p(s,x,z)\Phi(z)\,d z
    &\leq
    m_s\e^{\lambda_0s}\Phi(x)
   \langle \Phi(\cdot)^2  \mathbf{1}_{\mathbb{R}^d\backslash \overline B_R} \rangle_\rho \\
    &\leq
    \frac12\e^{\lambda_0s}\Phi(x).
\end{align*}
Subtracting this estimate from the ground-state identity,
we obtain
\[
    \int_{ \overline B_R}p(s,x,z)\Phi(z)\,d z
    \geq\frac12\e^{\lambda_0s}\Phi(x),
    \qquad x\in\R^d,
\]
which proves \eqref{eq:initialmass}.

Finally, the weighted symmetry in
Proposition~\ref{prop:kernel}\textup{(i)} implies
\begin{align*}
    \int_{ \overline B_R}\Phi(w)\rho(w)p(s,w,y)\,d w
    &=
    \rho(y)\int_{ \overline B_R}p(s,y,w)\Phi(w)\,d w\\
    &\geq
    \frac12\e^{\lambda_0s}\Phi(y)\rho(y),
    \qquad y\in\R^d.
\end{align*}
This proves \eqref{eq:terminalmass}.
The equality also shows that the second integral is finite.
By its construction, the radius $R$ is independent
of $x$ and $y$.
\end{proof}

We use the semigroup property to split the time interval into
three equal parts. Continuity and strict positivity give a lower
bound for the middle kernel on $\overline B_R\times\overline B_R$,
while Lemma~\ref{lem:mass} provides lower bounds for the two
remaining integrals.

\begin{proposition}\label{prop:bridge}
Fix $s>0$, and let $R\geq1$ be chosen as in Lemma~\ref{lem:mass}. Then
\begin{equation}\label{eq:kappa}
    \kappa_s
    :=\min_{(z,w)\in { \overline B_R} \times { \overline B_R}}
    \frac{\e^{-\lambda_0s}p(s,z,w)}
         {\Phi(z)\Phi(w)\rho(w)}
    >0.
\end{equation}
Moreover,
\begin{equation}\label{eq:at3s}
    p(3s,x,y)
    \geq
    \frac{\kappa_s}{4}\e^{3\lambda_0s}
    \Phi(x)\Phi(y)\rho(y),
    \qquad x,y\in\R^d.
\end{equation}
\end{proposition}

\begin{proof}
By Proposition~\ref{prop:kernel} and the continuity
and strict positivity of $\Phi$ and $\rho$, 
$$
    (z,w)\longmapsto
    \frac{\e^{-\lambda_0s}p(s,z,w)}
         {\Phi(z)\Phi(w)\rho(w)}
$$
is continuous and strictly positive on 
$ { \overline B_R} \times { \overline B_R}$.
It therefore attains a strictly positive minimum,
which proves \eqref{eq:kappa}.
In particular,
\begin{equation}\label{eq:bridge}
    p(s,z,w)
    \geq
    \kappa_s\e^{\lambda_0s}\Phi(z)\Phi(w)\rho(w),
    \qquad z,w\in { \overline B_R}.
\end{equation}

Fix $x,y\in\R^d$.
By Proposition~\ref{prop:kernel}\textup{(iv)}
and Tonelli's theorem,
\begin{align*}
    p(3s,x,y)
    &=\int_{\R^d}p(s,x,z)p(2s,z,y)\, d z\\
    &=\int_{\R^d}\int_{\R^d}
      p(s,x,z)p(s,z,w)p(s,w,y)\, d w\, d z.
\end{align*}
Tonelli's theorem applies because the integrand is
nonnegative, and the double integral is finite since
it equals $p(3s,x,y)$.

Restricting the integration to ${ \overline B_R}\times { \overline B_R}$ and using
\eqref{eq:bridge}, we obtain
\begin{align*}
    p(3s,x,y)
    &\geq
    \int_{ \overline B_R}\int_{ \overline B_R}
    p(s,x,z)p(s,z,w)p(s,w,y)\, d w\, d z\\
    &\geq
    \kappa_s\e^{\lambda_0s}
    \left(\int_{ \overline B_R} p(s,x,z)\Phi(z)\, d z\right)
    \left(\int_{ \overline B_R} \Phi(w)\rho(w)p(s,w,y)\, d w\right).
\end{align*}
Both terms are finite by Lemma~\ref{lem:mass}.
Applying the two lower bounds from Lemma~\ref{lem:mass} gives
\begin{align*}
    p(3s,x,y)
    &\geq
    \kappa_s\e^{\lambda_0s}
    \left(\frac12\e^{\lambda_0s}\Phi(x)\right)
    \left(\frac12\e^{\lambda_0s}\Phi(y)\rho(y)\right)\\
    &=\frac{\kappa_s}{4}\e^{3\lambda_0s}
      \Phi(x)\Phi(y)\rho(y).
\end{align*}
This proves \eqref{eq:at3s}. ${ \overline B_R}$ and the constant $\kappa_s$
are independent of $x$ and $y$.
\end{proof}

\section{Proof of Theorem \ref{thm:main}}
\label{sec:proof}
\begin{proof}
Fix $\tau>0$ and set $s=\tau/3$. We first establish the
global intrinsic lower bound and then use the ground-state
estimates to obtain the explicit spatial profile.

\smallskip

Step 1. (The lower bound at time $\tau$).
Choose $R$ by Lemma~\ref{lem:mass} and let $\kappa_s>0$
be given by \eqref{eq:kappa}. Proposition~\ref{prop:bridge} yields
\begin{equation}\label{eq:attau}
 p(\tau,x,y)\geq\frac{\kappa_s}{4}\e^{\lambda_0\tau}
 \Phi(x)\Phi(y)\rho(y),
 \qquad x,y\in\mathbb R^d.
\end{equation}

Once $\tau$ is fixed, the radius $R$ and the constant $\kappa_s$ are chosen independently of $x$ and $y$.

\smallskip

Step 2. (Extension to all $t\geq\tau$). For $t>\tau$, the semigroup property and
\eqref{eq:attau} give
\begin{align*}
 p(t,x,y)
 &=\int_{\mathbb R^d}p(t-\tau,x,z)p(\tau,z,y)\, d z\\
 &\geq\frac{\kappa_s}{4}\e^{\lambda_0\tau}\Phi(y)\rho(y)
       \int_{\mathbb R^d}p(t-\tau,x,z)\Phi(z)\, d z\\
 &=\frac{\kappa_s}{4}\e^{\lambda_0t}\Phi(x)\Phi(y)\rho(y).
\end{align*}
The last equality follows from the ground-state identity
in Proposition~\ref{prop:kernel}\textup{(ii)}. Together with
\eqref{eq:attau}, this proves
\begin{equation}\label{eq:groundlower}
 p(t,x,y)\geq\frac{\kappa_{\tau/3}}{4}\e^{\lambda_0t}
 \Phi(x)\Phi(y)\rho(y),
 \qquad t\geq\tau,\quad x,y\in\mathbb R^d.
\end{equation}
This proves \eqref{eq:main-intrinsic} with
$\widetilde C_\tau=\kappa_{\tau/3}/4$.

\smallskip

Step 3. (The explicit spatial profile). For $|x|,|y|\geq1$, Proposition~\ref{baondedofphi}
and the definition of $\rho$ imply
\begin{align*}
 \Phi(x)\Phi(y)\rho(y)
 &\geq c_-^2(|x||y|)^{-\gamma}\e^{-I(|x|)-I(|y|)}
 (1+|x|^\alpha)^{-b/(2\alpha)}
 (1+|y|^\alpha)^{b/(2\alpha)-1}\\
 &=c_-^2
 \left(\frac{1+|y|^\alpha}{1+|x|^\alpha}\right)^{b/(2\alpha)}
 \frac{(|x||y|)^{-\gamma}}{1+|y|^\alpha}
 \e^{-I(|x|)-I(|y|)}.
\end{align*}
Substitution into \eqref{eq:groundlower} proves
\eqref{eq:main} with
\begin{equation}\label{eq:Ctau}
 C_\tau:=\frac{c_-^2\kappa_{\tau/3}}4>0.
\end{equation}
This completes the proof.
\end{proof}

\makeatletter
\let\savedtocwrite\@tocwrite
\let\@tocwrite\@gobbletwo
\section*{AI usage disclosure}
\let\@tocwrite\savedtocwrite
\makeatother

The author used ChatGPT for editorial purposes, i.e.\,fixing typos, improving language, and carrying out the final editorial review of the paper. 

\bigskip

\noindent
\text{Data availability.} This article has no associated datasets.

\end{document}